\documentclass{amsart}
\usepackage[all]{xy}
\usepackage{amsmath,amssymb,dsfont,mathtools,enumitem}
\usepackage[pdftitle={Localisation and colocalisation of triangulated categories at thick subcategories},
  pdfauthor={Hvedri Inassaridze, Tamaz Kandelaki, Ralf Meyer},
  pdfsubject={Mathematics},
  pdfkeywords={triangulated category, localisation, derived functor},
]{hyperref}
\usepackage[lite]{amsrefs}
\usepackage{microtype}

\newcommand*{\MRref}[2]{ \href{http://www.ams.org/mathscinet-getitem?mr=#1}{MR \textbf{#1}}}

\newtheorem{thm}{Theorem}[section]
\newtheorem{cor}[thm]{Corollary}
\newtheorem{lem}[thm]{Lemma}
\newtheorem{prop}[thm]{Proposition}
\theoremstyle{definition}
\newtheorem{defn}[thm]{Definition}
\theoremstyle{remark}

\numberwithin{equation}{section}

\DeclareMathOperator{\im}{im}

\newcommand*{\Tri}{\mathcal T}
\newcommand*{\Sub}{\mathcal E}
\newcommand*{\Qeq}{\textup{we}_\Sub}
\newcommand*{\Loc}{\Tri/\Sub}
\newcommand*{\Colo}{\Tri/\Sub^\bot}
\newcommand*{\Cat}{\mathcal C}
\newcommand*{\Right}{\mathbb R}
\newcommand*{\Rightco}{\mathbb R^\bot}

\newcommand*{\Wo}[1]{#1 \downarrow \Qeq}
\newcommand*{\Subo}[1]{\Sub \downarrow#1}
\newcommand*{\Tria}[1]{\triangle_\Sub #1}

\newcommand*{\inOb}{\mathrel{\in\in}\nobreak}

\newcommand*{\Q}{\mathbb Q}

\newcommand*{\Ab}{\mathfrak{Ab}}

\newcommand*{\Cst}{\textup{C}^*}
\newcommand*{\K}{\textup K}
\newcommand*{\op}{\textup{op}}
\newcommand*{\KK}{\textup{KK}}
\newcommand*{\id}{\textup{id}}

\newcommand*{\nb}{\nobreakdash}
\newcommand*{\defeq}{\mathrel{\vcentcolon=}}
\newcommand*{\blank}{\textup{\textvisiblespace}}

\begin{document}
\title[Localisation and colocalisation]{Localisation and colocalisation of triangulated categories at thick subcategories}
\author{Hvedri Inassaridze}
\author{Tamaz Kandelaki}
\address{Hvedri Inassaridze, Tamaz Kandelaki: A.~Razmadze Mathematical Institute, M.~Alexidze Street~1, Tbilisi 0193, Georgia}
\email{inassari@gmail.com}
\email{tam.kandel@gmail.com}

\author{Ralf Meyer}
\address{Ralf Meyer: Mathematisches Institut and\\Courant Centre ``Higher order structures,'' Georg-August Universit\"at G\"ottingen, Bunsenstra{\ss}e 3--5, 37073 G\"ottingen, Germany}
\email{rameyer@uni-math.gwdg.de}

\thanks{This research was supported by the Volkswagen Foundation (Georgian--German Non-Commutative Partnership).  The third author was supported by the German Research Foundation (Deutsche Forschungsgemeinschaft (DFG)) through the Institutional Strategy of the University of G\"ottingen.}


\begin{abstract}
  Given a thick subcategory of a triangulated category, we define a colocalisation and a natural long exact sequence that involves the original category and its localisation and colocalisation at the subcategory.  Similarly, we construct a natural long exact sequence containing the canonical map between a homological functor and its total derived functor with respect to a thick subcategory.
\end{abstract}
\maketitle

\section{Introduction}
\label{sec:intro}

Throughout this article, \(\Tri\) is a triangulated category and~\(\Sub\) is a thick subcategory.  The \emph{localisation} of~\(\Tri\) at~\(\Sub\) is a triangulated category \(\Loc\) with a triangulated functor \(L\colon \Tri\to\Loc\) such that \(L|_\Sub \cong0\), and such that any other triangulated functor with this property factors uniquely through~\(L\).  This universal property determines the localisation uniquely up to isomorphism if it exists.  The localisation is constructed in~\cite{Neeman:Triangulated}, disregarding set-theoretic issues.  These also appear in our context.  We must assume that certain colimits exist, which follows, for instance, if the triangulated categories in question are essentially small.

Let \(A[n]\) denote the \(n\)\nb-fold suspension of~\(A\) and let \(\Tri_n(A,B) \defeq \Tri(A,B[n]) \cong \Tri(A[-n],B)\), and similarly for~\(\Loc\).  We are going to embed the localisation functor \(L\colon \Tri\to\Loc\) into a natural long exact sequence
\begin{multline}
  \label{eq:localisation_exact_sequence}
  \dotsb
  \to \Colo_1(A,B)
  \to \Tri_1(A,B)
  \xrightarrow{L} \Loc_1(LA,LB)
  \\\to \Colo_0(A,B)
  \to \Tri_0(A,B)
  \xrightarrow{L} \Loc_0(LA,LB)
  \to \dotsb,
\end{multline}
with \(\Colo_n(A,B) \defeq \Colo(A,B[n]) \cong \Colo(A[-n],B)\).  We call~\(\Colo\) the \emph{colocalisation} of~\(\Tri\) at~\(\Sub\).  The naturality of~\eqref{eq:localisation_exact_sequence} means that \(\Colo(A,B)\) is a bifunctor on~\(\Tri\) contravariant in the first and covariant in the second variable, and that the maps in the above exact sequence are natural transformations.  Using the natural transformation \(\Colo\to\Tri\) in~\eqref{eq:localisation_exact_sequence} and the \(\Tri\)\nb-bimodule structure, \(\Colo\) is equipped with an associative composition product.  It is not a category because it lacks identity morphisms.

Given an Abelian category~\(\Cat\) and a homological functor \(F\colon \Tri\to\Cat\), its \emph{right localisation} at~\(\Sub\) is a homological functor \(\Right F\colon \Tri\to\Cat\) together with a natural transformation \(F\Rightarrow \Right F\), such that~\(\Right F\) vanishes on~\(\Sub\) and such that any other natural transformation \(F\Rightarrow G\) for a homological functor~\(G\) with \(G|_\Sub=0\) factors uniquely through \(F\Rightarrow \Right F\).  Left localisations are defined similarly, but they will not be treated here.

If~\(F\) is a homological functor to the category of Abelian groups, then our construction of the colocalisation functor also provides a long exact sequence
\begin{multline}
  \label{eq:derived_exact_sequence}
  \dotsb
  \to \Rightco F_1(A)
  \to F_1(A)
  \to \Right F_1(A)
  \\\to \Rightco F_0(A)
  \to F_0(A)
  \to \Right F_0(A)
  \to \dotsb,
\end{multline}
which is functorial in~\(A\) and~\(F\); that is, a natural transformation of homological functors \(F\Rightarrow F'\) and an arrow \(A\to A'\) induce a chain map from~\eqref{eq:derived_exact_sequence} for \(F_*(A)\) to~\eqref{eq:derived_exact_sequence} for \(F'_*(A')\) in a functorial manner.

We call a homological functor \(F\colon \Tri\to\Ab\) \emph{local} if the natural map \(F\Rightarrow\Right F\) is invertible, and \emph{colocal} if the map \(\Rightco F\Rightarrow F\) is invertible.  As we shall explain, local homological functors correspond to homological functors \(\Loc\to\Ab\), while colocal homological functors correspond to homological functors \(\Sub\to\Ab\).  Thus~\eqref{eq:derived_exact_sequence} decomposes homological functors on~\(\Tri\) into homological functors on \(\Loc\) and~\(\Sub\).

Dual statements apply to cohomological functors \(\Tri\to\Ab\) because we may view them as homological functors \(\Tri^\op\to\Ab\) on the opposite category~\(\Tri^\op\).

The exact sequence~\eqref{eq:localisation_exact_sequence} is easy to construct in the special situation of a \emph{complementary pair of thick subcategories} in the notation of~\cite{Meyer-Nest:BC}.  (In~\cite{Neeman:Triangulated}, this is called a Bousfield localisation.)  Recall that~\((\Sub^\bot,\Sub)\) is a complementary pair of thick subcategories if \(\Tri(\Sub^\bot,\Sub)=0\) and for any object~\(B\) of~\(\Tri\) there is an exact triangle
\begin{equation}
  \label{eq:triangle_complementary}
  L^\bot B \to B \to LB \to L^\bot B[1]
\end{equation}
with \(LB\inOb\Sub^\bot\), \(L^\bot B\inOb\Sub\).  Here we write~\(\inOb\) for objects of categories, as opposed to~\(\in\) for morphisms of categories.  The triangle~\eqref{eq:triangle_complementary} is unique up to canonical isomorphism and depends functorially on~\(B\).  Since
\begin{alignat*}{2}
  \Loc(A,B) &\cong \Tri(A,LB) \cong \Tri(LA,LB),&\qquad
  \Right F(B) &\cong F(LB),\\\shortintertext{we may define}
  \Colo(A,B) &\defeq \Tri(A,L^\bot B),&\qquad
  \Rightco F(B) &\defeq F(L^\bot B).
\end{alignat*}
Then the exact sequences \eqref{eq:localisation_exact_sequence} and~\eqref{eq:derived_exact_sequence} follow by applying the homological functors \(\Tri(\blank,B)\) and~\(F\) to the exact triangle~\eqref{eq:triangle_complementary}.

Without~\(\Sub^\bot\), we cannot single out a unique exact triangle as in~\eqref{eq:triangle_complementary}.  Instead, we construct the localisation and the colocalisation as colimits of \(\Tri(A,D)\) and \(\Tri(A,C)\), where \(C\) and~\(D\) are indexed by the category of exact triangles \(C\to B\to D\to C[1]\) with \(C\inOb\Sub\).  The main issue is to prove that this category of exact triangles is filtered, so that the colimit preserves exact sequences.

A typical example where the general construction of the colocalisation is useful is considered in~\cite{Inassaridze-Kandelaki-Meyer:Finite_Torsion_KK}.  There \(\Tri(A,B) = \KK(A,B)\) is Kasparov's bivariant \(\K\)\nb-theory for \(\Cst\)\nb-algebras \(A\) and~\(B\), and \(\Loc(A,B) = \KK(A,B) \otimes \Q\).  The colocalisation provides a torsion variant of \(\KK\)-theory.  It can be shown that \(\Loc\) is the localisation of~\(\KK\) at a suitable thick subcategory.  It is unclear, however, whether this is part of a complementary pair of thick subcategories.

This article is organised as follows.  Section~\ref{sec:localisation} recalls the well-known construction of the localisation and its basic properties.   Similar techniques are used in Section~\ref{sec:colocalisation} to define and study colocalisations.  Section~\ref{sec:exact_sequence} establishes the exact sequences \eqref{eq:localisation_exact_sequence} and~\eqref{eq:derived_exact_sequence}.  Each section uses a particular filtered category.

\section{Localisation}
\label{sec:localisation}

We fix a triangulated category~\(\Tri\) and a thick subcategory~\(\Sub\) throughout.

The \emph{cone} of a morphism \(f\colon A\to B\) in~\(\Tri\) is the object~\(C\) in an exact triangle
\[
A\xrightarrow{f} B\to C \to A[1]
\]
A morphism~\(f\) in~\(\Tri\) is an \emph{\(\Sub\)\nb-weak equivalence} if its cone belongs to~\(\Sub\).  Let \(\Qeq(A,B)\) be the set of \(\Sub\)\nb-weak equivalences from~\(A\) to~\(B\).

Recall that an object of~\(\Tri\) becomes zero in~\(\Loc\) if and only if it belongs to~\(\Sub\), and an arrow in~\(\Tri\) becomes invertible in~\(\Loc\) if and only if it is an \(\Sub\)\nb-weak equivalence (see~\cite{Neeman:Triangulated}).  Thus \(\Qeq\) is closed under composition and contains all isomorphisms in~\(\Tri\).  Even more, \(\Qeq\) has the two-out-of-three and even the two-out-of-six property:

\begin{lem}
  \label{lem:two_six}
  Let \(f\), \(g\) and~\(h\) be composable morphisms in~\(\Tri\).  If \(fg\) and \(gh\) are \(\Sub\)\nb-weak equivalences, so are \(f\), \(g\), \(h\), and \(fgh\).

  If two of \(f\), \(g\), and \(fg\) are \(\Sub\)\nb-weak equivalences, so is the third.
\end{lem}

\begin{proof}
  The assumption means that \(L(fg)\) and~\(L(gh)\) are invertible in~\(\Loc\).  This implies that~\(L(g)\) has both a left and a right inverse, so that \(L(g)\) is invertible.  Then \(L(f)\) and \(L(h)\) must also be invertible, and finally \(L(fgh)\) is invertible.  The second statement follows from the first one by considering the special cases where one of \(f\), \(g\) or~\(h\) is an identity (see also \cite{Neeman:Triangulated}*{Lemma 1.5.6} for a more direct proof of the two-out-of-three property).
\end{proof}

The construction of the localisation~\(\Loc\) in~\cite{Neeman:Triangulated} shows that its underlying category is isomorphic to the localisation of~\(\Tri\) at~\(\Qeq\) in the category theoretic sense of~\cite{Gabriel-Zisman:Fractions}.  To prepare for this, it is shown that~\(\Qeq\) allows a \emph{calculus of left fractions} in the sense of \cite{Gabriel-Zisman:Fractions}*{\S2.2}, that is, it satisfies the following conditions (LF1)--(LF3):

\begin{lem}
  \label{lem:trianleft}
  Let~\(\Sub\) be a thick subcategory of a triangulated category~\(\Tri\).
  \begin{enumerate}[label=\textup{(LF\arabic{*})}]
  \item Identities belong to~\(\Qeq\), and~\(\Qeq\) is closed under composition.
  \item For each pair of morphisms \(B\xleftarrow{s}A\xrightarrow{f}C\) with \(s\in\Qeq\) there are \(g\in\Tri\) and \(t\in\Qeq\) with \(gs=tf\), that is, there is a commuting square
    \[
    \xymatrix{
      A\ar[r]^f\ar[d]_s^\sim& C\ar[d]^t_\sim\\
      B\ar[r]^g &D.}
    \]
  \item If \(f,g\colon A\rightrightarrows B\) are parallel morphisms in~\(\Tri\) and \(s\colon A'\rightarrow A\) is in~\(\Qeq\) and satisfies \(fs=gs\), then there is \(t\colon B\rightarrow B'\) in~\(\Qeq\) with \(tf=tg\).
  \end{enumerate}
\end{lem}

\begin{proof}
  (LF1) is obvious, and (LF2) follows from \cite{Neeman:Triangulated}*{Lemma 1.5.8} and the Octahedral Axiom.  We only have to prove (LF3).  We construct the following commuting diagram:
  \[
  \xymatrix{
    A'\ar[r]^s_\sim \ar[dr]_0&
    A\ar[r]\ar[d]^{f-g}&
    C\ar[d]\\
    &B\ar @{=}[r]&B\ar[d]^t_\sim&\\
    &&B'
  }
  \]
  First we embed \(s\colon A'\to A\) in an exact triangle \(A'\xrightarrow{s} A\to C \to A'[1]\).  Then \(C\inOb\Sub\) because \(s\in\Qeq\).  Since \((f-g)\circ s=0\) and \(\Tri(\blank,B)\) is cohomological, we may factor \(f-g\) through the map \(A\to C\).  We embed the resulting map \(C\to B\) into an exact triangle \(C\to B\xrightarrow{t} B'\to C[1]\).  Since \(C\inOb\Sub\), \(t\in\Qeq\).  By construction \(t\circ (f-g)=0\).
\end{proof}

Dually, there is a calculus of right fractions for the same reasons.

It is useful for us to rewrite \(\Loc(A,B)\) as a filtered colimit.  For a fixed object \(B\inOb\Tri\), we consider the category~\(\Wo{B}\) whose objects are arrows \(B\to C\) in~\(\Qeq\) and whose arrows are commuting triangles:
\[
\xymatrix@C-1.5em{C\ar[rr]^r&&C'\\
  &B\ar[lu]^s\ar[ru]_t&}
\]
where \(s\) and~\(t\) are morphisms in~\(\Qeq\) and~\(r\) is a morphism in~\(\Tri\).  By the two-out-of-three property, we also get \(r\in\Qeq\).

Recall that a category is called \emph{filtered} if it satisfies the following conditions:
\begin{enumerate}[label=\textup{(F\arabic{*})}]
\item for two objects \(E\) and~\(D\) there are an object~\(H\) and morphisms \(f\colon E\to H\) and \(g\colon D\to H\);
\item for a pair of parallel morphisms \(h,j\colon A\rightrightarrows B\) there is a morphism \(l\colon B\to F\) with \(lh=lj\).
\end{enumerate}

Since~\(\Qeq\) allows a calculus of left fractions by Lemma~\ref{lem:trianleft}, the category~\(\Wo{B}\) is filtered for each~\(B\).  (The proof that (LF1)--(LF3) imply (F1) and~(F2) is an easy exercise.)

The map \(s\mapsto \Tri(A,C)\) for \((s\colon B\to C) \inOb \Wo{B}\) defines a covariant functor from \(\Wo{B}\) to the category of sets.  The colimit of this diagram yields the space of arrows \(\Loc(A,B)\) in the localisation of~\(\Tri\) at~\(\Qeq\):
\begin{equation}
  \label{eq:localisation}
  \Loc(A,B) \defeq \varinjlim_{(s\colon B\to C)\inOb (\Wo{B})} \Tri(A,C).
\end{equation}
The existence of these colimits for all \(A\) and~\(B\) ensures that the localisation exists.  Recall that any small diagram of Abelian groups has a colimit.  If~\(\Tri\) is (essentially) small, then so is~\(\Wo{B}\), and hence the colimit in~\eqref{eq:localisation} exists.  There are many examples where~\(\Tri\) is large but the colimits in~\eqref{eq:localisation} exist for different reasons.

For instance, if~\(\Sub\) is part of a complementary pair of thick subcategories \((\Sub^\bot,\Sub)\) in the notation of~\cite{Meyer-Nest:BC}, then the arrow \(B\to LB\) in~\eqref{eq:triangle_complementary} is a final object of~\(\Wo{B}\).  Thus the colimit above exists and reduces to \(\Loc(A,B) \cong \Tri(A,LB)\).  Conversely, if~\(\Wo{B}\) has a final object for each~\(B\), then \((\Sub^\bot,\Sub)\) is a complementary pair of thick subcategories.

\begin{lem}
  \label{lem:Wo_B_functor}
  The inductive system \((C)_{s\colon B\to C\inOb\Wo{B}}\) depends functorially on~\(B\).
\end{lem}

\begin{proof}
  Let \(f\colon B\to B'\) be a morphism in~\(\Tri\).  By~(LF1), we associate to an element \(s\colon B\to C\) an element \(t\colon B'\to C'\) of~\(\Wo{B'}\) and a map \(g\colon C\to C'\) with \(gs=tf\).  It follows from the calculus of left fractions that this defines a morphism of inductive systems that does not depend on our auxiliary choices of \(t\) and~\(g\).

  Alternatively, morphisms of inductive systems \((C_s)\to (C'_{s'})\) correspond bijectively to natural transformations \(\Tri(X,(C_s)) \to \Tri(X,(C'_{s'}))\).  Since \(\Tri(X,(C_s)) = \Loc(X,B)\), our statement is equivalent to the existence of a natural composition product
  \[
  \Loc(X,B)\times\Tri(B,B') \to \Loc(X,B').
  \]
  This exists because~\(\Loc\) is a category and the maps \(\Tri\to\Loc\) are a functor.
\end{proof}

\begin{defn}
  \label{def:localisation_functor}
  Let \(F\colon \Tri\to\Ab\) be a homological functor to the category of Abelian groups.  Its \emph{right derived functor} or \emph{right localisation} \(\Right F\colon \Tri\to\Ab\) is defined by
  \[
  \Right F(B) \defeq \varinjlim_{(s\colon B\to C)\inOb\Wo{B}} F(C).
  \]
  This is a functor by Lemma~\ref{lem:Wo_B_functor}.  The maps \(s\colon B\to C\) provide a natural transformation \(F\Rightarrow \Right F\).
\end{defn}

\begin{thm}
  \label{thm:localisation_homological}
  Let~\(\Tri\) be a triangulated category, \(\Sub\) a thick subcategory, and \(F\colon \Tri \to \Ab\) a homological functor.  Then the right localisation \(\Right F\colon \Tri \to \Ab\) is homological.
\end{thm}

\begin{proof}
  Let \(B'\xrightarrow{f} B\xrightarrow{g} B''\to B'[1]\) be an exact triangle.  We must show that
  \[
  \Right F(B')
  \xrightarrow{\Right F(f)} \Right F(B)
  \xrightarrow{\Right F(g)}\Right F(B'')
  \]
  is an exact sequence of groups.  Functoriality of \(\Right F\) already implies \(\im \Right F(f) \subseteq \ker \Right F(g)\).  It remains to prove \(\im \Right F(f)\supseteq \ker \Right F(g)\).

  Let \(s\in \Qeq(B,C)\) and \(x\in F(C)\) represent an element in \(\ker \Right F(g)\).  By the definition of the functoriality of \(\Right F\), this means that there is a commuting diagram
  \[
  \xymatrix{C\ar[r]^{\hat{g}}&C''\\
    B\ar[r]^g\ar[u]^s&B''\ar[u]_{s''}
  }
  \]
  with \(s''\in\Qeq\) and \(F(\hat{g})(x)=0\).  By the properties of triangulated categories, we may complete this to a morphism of exact triangles
  \[
  \xymatrix{C'\ar[r]^{\hat{f}}&C\ar[r]^{\hat{g}}&C''\ar[r]&C'[1]\\
    B'\ar[u]^{s'}\ar[r]^f&B\ar[r]^g\ar[u]^s&B''\ar[u]^{s''}\ar[r]&B'[1]\ar[u].}
  \]
  Since \(s,s''\in\Qeq\), we also get \(s'\in\Qeq\) from the Octahedral Axiom.  Since~\(F\) is homological, there is an element \(x'\in F(C')\) with \(F(\hat{f})(x')=x\).  Therefore, the class of the pair \((s',x')\) in \(\Right F(B')\) maps to the class of \((s,x)\).
\end{proof}

Recall that the category~\(\Tri^\op\) is again triangulated, with new suspension automorphism \(A\mapsto A[-1]\) and the same exact triangles: \(A[1]\leftarrow C \leftarrow B\leftarrow A\) is an exact triangle in~\(\Tri^\op\) if \(A\to B\to C\to A[1]\) is an exact triangle in~\(\Tri\).  Thus we may view a cohomological functor \(\Tri\to\Ab\) as a homological functor \(\Tri^\op\to\Ab\).  Definition~\ref{def:localisation_functor} applied to a homological functor \(\Tri^\op\to\Ab\) yields a right localisation for the original cohomological functor \(\Tri\to\Ab\).

Theorem~\ref{thm:localisation_homological} may be extended to the case where~\(F\) is a homological functor with values in an Abelian category~\(\Cat\) with \emph{exact filtered colimits}.  However, some assumption on~\(\Cat\) seems necessary.  In particular, it is unclear how to treat homological functors \(\Tri\to\Ab^\op\) because filtered colimits in~\(\Ab^\op\) are filtered limits in~\(\Ab\), and these need not be exact.

By definition, the functor \(B\mapsto \Loc(A,B)\) is the right localisation of \(B\mapsto \Tri(A,B)\).  Therefore, this functor is homological by Theorem~\ref{thm:localisation_homological}.  Since filtered colimits are exact, the functor \(A\mapsto \Loc(A,B)\) is cohomological by~\eqref{eq:localisation}.  Of course, this also follows from the stronger statement that~\(\Loc\) is a triangulated category.

We may also apply the construction above to get the localisation of~\(\Tri^\op\) at~\(\Sub^\op\).  The opposite category of this localisation provides another model for~\(\Loc\) that is based on a calculus of \emph{right} fractions instead of left fractions.  Both constructions agree because both localisations share the same universal property.

\begin{prop}
  \label{pro:local_functors}
  Let \(F\colon \Tri\to\Ab\) be a homological functor.  The following assertions are equivalent:
  \begin{enumerate}[label=\textup{(\alph{*})}]
  \item the natural transformation \(F\Rightarrow \Right F\) is invertible;
  \item \(F(E)\cong0\) for all \(E\inOb\Sub\);
  \item \(F(s)\) is invertible for all \(s\in\Qeq\);
  \item \(F\) factors through a homological functor \(\Loc\to\Ab\).
  \end{enumerate}
  Furthermore, \(\Right F\) always satisfies these equivalent conditions.
\end{prop}

\begin{proof}
  If \(E\inOb\Sub\), then the zero map \(E\to 0\) is an \(\Sub\)\nb-weak equivalence.  It is a final object in \(\Wo{E}\), so that \(\Right F(E) = F(0)=0\).  Thus \(\Right F(E)=0\) always satisfies~(b), and hence~(a) implies~(b).  If \(s\in\Qeq\), then~\(s\) is part of an exact triangle \(A\xrightarrow{s} B\to E\to A[1]\) with \(E\inOb\Sub\).  The long exact sequence for~\(F\) applied to this exact triangle shows that \(F_*(s)\) is invertible if and only if \(F_*(E)\cong0\).  Since \(\Qeq\) and~\(\Sub\) are closed under suspensions, this yields (b)\(\iff\)(c).  Since~\(\Loc\) is the localisation of~\(\Tri\) at~\(\Qeq\) in the sense of category theory, a functor on~\(\Tri\) factors through~\(\Loc\) if and only if it maps all arrows in~\(\Qeq\) to invertible arrows.  Furthermore, the functor on~\(\Loc\) induced in this way is again homological if~\(F\) was (see~\cite{Neeman:Triangulated}).  Thus (c)\(\iff\)(d).  Finally, (c)~implies that the maps \(F(B)\to F(C)\) are invertible for all \(s\colon B\to C\) in~\(\Wo{B}\).  Thus~(c) implies~(a).
\end{proof}

There is an analogous result for cohomological functors.

A (co)homological functor with the equivalent properties in Proposition~\ref{pro:local_functors} is called \emph{local}.  Condition~(d) means that local (co)homological functors \(\Tri\to\Ab\) are equivalent to local (co)homological functors \(\Loc\to\Ab\).

\begin{prop}
  \label{pro:localisation_universal}
  The localisation \(\Right F\) is the universal local homological functor on~\(\Tri\) equipped with a natural transformation \(F\Rightarrow \Right F\): if~\(G\) is any local homological functor on~\(\Tri\), then there is a natural bijection between natural transformations \(F\Rightarrow G\) and natural transformations \(\Right F\Rightarrow G\).
\end{prop}

This universal property characterises~\(\Right F\) uniquely up to natural isomorphism.

\begin{proof}
  If~\(G\) is local, then a natural transformation \(\Phi\colon F\Rightarrow G\) induces a natural transformation \(\Phi'\colon \Right F\Rightarrow \Right G\cong G\).  The product of~\(\Phi'\) with the natural transformation \(\Psi\colon F\Rightarrow \Right F\) is again~\(\Phi\), and~\(\Phi'\) is the only natural transformation \(\Right F\Rightarrow G\) with \(\Phi'\circ\Psi=\Phi\).
\end{proof}

\section{Colocalisation}
\label{sec:colocalisation}

Let \(\Subo{B}\) be the category, whose objects are arrows \(f\colon E\to B\) with \(E\inOb\Sub\) and whose arrows are commuting triangles
\[
\xymatrix@R-2em{E\ar[dd]_r\ar[dr]^f\\&B.\\E'\ar[ru]_{f'}}
\]

\begin{lem}
  \label{lem:Subo_filtered}
  The category \(\Subo{B}\) is filtered for all \(B\inOb\Tri\).
\end{lem}

\begin{proof}
  The first axiom~(F1) of a filtered category follows easily because~\(\Sub\) is additive: any pair of maps \(E_1\to B\), \(E_2\to B\) is dominated by \(E_1\oplus E_2\to B\).  To verify~(F2), we must equalise a diagram of the form
  \begin{equation}
    \label{parallel1}
    \begin{gathered}
      \xymatrix@R-2em{E_1\ar@<-0.5ex>[dd]_r\ar@<+0.5ex>[dd]^{r'}\ar[dr]^{f_1}\\&B.\\E_2\ar[ru]_{f_2}}
    \end{gathered}
  \end{equation}
  Embed \(r-r'\) in an exact triangle \(E_1\xrightarrow{r-r'} E_2\to E_3\to E_1[1]\).  Since \(f_2\circ (r-r')=0\) and \(\Tri(\blank,B)\) is cohomological, we may factor~\(f_2\) through the map \(E_2\to E_3\).  This yields a commuting diagram
  \[
  \xymatrix@R-.5em{
    E_1\ar[dr]^{f_1} \ar@<-0.5ex>[d]_-{r} \ar@<0.5ex>[d]^-{r'}\\
    E_2\ar[r]^{f_2}\ar[d]&B,\\
    E_3\ar[ru]_{f_3}
  }
  \]
  which equalises the diagram~\eqref{parallel1} in \(\Subo{B}\).
\end{proof}

\begin{defn}
  \label{def:colocalisation}
  The \emph{colocalisation} of~\(\Tri\) at~\(\Sub\) is defined by
  \[
  \Colo(A,B) \defeq \varinjlim_{(s\colon C\to B) \inOb\Subo{B}} \Tri(A,C).
  \]
\end{defn}

As in the discussion after~\eqref{eq:localisation}, this colimit exists if~\(\Tri\) is small or if~\(\Subo{B}\) has a final element.  The latter happens for all~\(B\) if and only if~\(\Sub\) is part of a complementary pair of thick subcategories \((\Sub^\bot,\Sub)\).  In this case, the map \(L^\bot B\to B\) in~\eqref{eq:triangle_complementary} is a final object in~\(\Subo{B}\), so that \(\Colo(A,B) \cong \Tri(A,L^\bot B)\).

We will see below that Definition~\ref{def:colocalisation} leads to an exact sequence as in~\eqref{eq:localisation_exact_sequence}.

The naturality of the inductive system \((C)_{s\inOb\Subo{B}}\) is trivial: an arrow \(f\colon B\to B'\) in~\(\Tri\) induces a morphism of inductive systems
\[
(C)_{(s\colon C\to B)\inOb\Subo{B}}\to (C')_{(s'\colon C'\to B')\inOb\Subo{B'}},
\]
which maps \(s\colon C\to B\) to \(f\circ s\colon C\to B'\) and acts identically on~\(C\).  As a result, \(B\mapsto \Colo(A,B)\) is a bifunctor that is contravariant in~\(A\) and covariant in~\(B\).

The maps \(s\colon C\to B\) provide a natural natural transformation
\[
\Colo(A,B)\to \Tri(A,B).
\]
We may describe elements of~\(\Colo(A,B)\) as diagrams \(A \xrightarrow{f} \tilde{B} \xrightarrow{s} B\) with \(\tilde{B}\inOb\Sub\).  The natural map to \(\Tri(A,B)\) maps this diagram to \(sf\colon A\to B\).

The naturality \(\Tri(B,C)\times \Colo(A,B)\to\Colo(A,C)\) and the natural map \(\Colo(B,C)\to\Tri(B,C)\) provide a multiplication
\[
\Colo(B,C)\times\Colo(A,B) \to \Colo(A,C).
\]
The product of \(B \xrightarrow{f_1}\tilde{C} \xrightarrow{s_1} C\) and
\(A \xrightarrow{f_2}\tilde{B} \xrightarrow{s_2} B\) is
\[
A \xrightarrow{f_1s_2f_2}\tilde{C} \xrightarrow{s_1} C
\sim A \xrightarrow{f_2}\tilde{B} \xrightarrow{s_1f_1s_2} C,
\]
where~\(\sim\) denotes equivalence in the inductive limit \(\Colo(A,C)\).  That is, we get the same multiplication on \(\Colo\) if we use the right multiplication map
\[
\Colo(B,C)\times \Tri(A,B)\to\Colo(A,C)
\]
and the natural map \(\Colo(A,B)\to\Tri(A,B)\).  It is also straightforward to see that the multiplication on \(\Colo\) is associative.  However, \(\Colo\) has no identity maps, so that it is not a category.

\begin{defn}
  \label{def:colocalisation_functor}
  Given a homological functor \(F\colon \Tri\to\Ab\) to the category of Abelian groups, we define its \emph{right colocalisation} at~\(\Sub\) by
  \[
  \Rightco F(B) \defeq  \varinjlim_{(s\colon C\to B) \inOb\Subo{B}} F(C).
  \]
\end{defn}

The naturality of the inductive system \((C)_{s\inOb\Subo{B}}\) implies
that~\(\Rightco F\) is a functor on~\(\Tri\).  The maps \(s\colon C\to
B\) induce a natural map \(\Rightco F(B) \to F(B)\).  A natural
transformation \(F\Rightarrow F'\) clearly induces a natural
transformation \(\Rightco F\Rightarrow \Rightco F'\).

\begin{thm}
  \label{the:colocalisation_homological}
  If~\(F\) is a homological functor on~\(\Tri\), then so is~\(\Rightco F\).
\end{thm}

\begin{proof}
  Let \(B'\xrightarrow{f} B\xrightarrow{g} B''\to B'[1]\) be an exact triangle.  We must show that
  \[
  \Rightco F(B')
  \xrightarrow{\Rightco F(f)} \Rightco F(B)
  \xrightarrow{\Rightco F(g)} \Rightco F(B'')
  \]
  is an exact sequence.  Functoriality already implies \(\ker \Rightco F(g) \supseteq \im \Rightco F(f)\).  It remains to show \(\ker \Rightco F(g)\subseteq \im \Rightco F(f)\).

  Represent an element of \(\Rightco F(B)\) by a pair \((s,x)\) with \(E\inOb\Sub\), \(s\colon E\to B\) and \(x\in F(E)\).  If \(\Rightco F(g)(s,x)=0\), then there is a commuting diagram
  \[
  \xymatrix{
    E \ar[r]^{\hat{g}}\ar[d]^{s}&
    E''\ar[d]^{s''}\\
    B\ar[r]^{g}&B''
  }
  \]
  with \(E''\inOb\Sub\) such that \(F(\hat{g})(x)=0\).  We may embed the above commuting square into a morphism of exact triangles
  \[
  \xymatrix{
    E' \ar[r]^{\hat{f}}\ar[d]^{s'}&
    E \ar[r]^{\hat{g}}\ar[d]^{s}&
    E''\ar[d]^{s''} \ar[r]&
    E'[1] \ar[d]^{s'[1]}\\
    B' \ar[r]^{f}&B\ar[r]^{g}&B''\ar[r]&B'[1]
  }
  \]
  Since~\(F\) is homological and \(F(\hat{g})(x)=0\), there is \(x'\in F(E')\) with \(x=F(\hat{f})(x')\).  Thus the class of \((s',x')\) in \(\Rightco F(B')\) is a pre-image for the class of \((s,x)\).
\end{proof}

\begin{thm}
  \label{thm:colocalisation_bifunctor}
  The map \((A,B)\mapsto \Colo(A,B)\) defines a bifunctor that is cohomological in the first and homological in the second variable.
\end{thm}

\begin{proof}
  Since \(\Colo(A,\blank)\) is the right colocalisation of
  \(\Tri(A,\blank)\), Theorem~\ref{the:colocalisation_homological} shows
  that \(\Colo(A,B)\) is homological in the second variable.  It is
  cohomological in the first variable because filtered colimits are
  exact.
\end{proof}

\begin{defn}
  \label{def:colocal}
  We call a (co)homological functor \(F\colon \Tri\to\Ab\) \emph{colocal} if the natural transformation \(\Rightco F\to F\) is invertible.
\end{defn}

\begin{prop}
  \label{pro:colocal_functors}
  Let \(F\colon \Sub\to\Ab\) be a homological functor.  Then there is a unique colocal homological functor \(\bar{F}\colon \Tri\to\Ab\) that extends~\(F\).  Thus colocal homological functors \(\Tri\to\Ab\) are essentially equivalent to homological functors \(\Sub\to\Ab\).

  Furthermore, \(\Rightco G\) is colocal for any homological functor \(G\colon \Tri\to\Ab\).  The natural transformation \(\Rightco G\Rightarrow G\) is universal among natural transformations from colocal functors to~\(G\).
\end{prop}

\begin{proof}
  We extend \(F\colon \Sub\to\Ab\) by
  \[
  \bar{F}(B) \defeq \varinjlim_{(s\colon C\to B) \inOb\Subo{B}} F(C).
  \]
  The proof of Theorem~\ref{the:colocalisation_homological} only needs that~\(F|_\Sub\) is a homological functor.  Hence it shows that~\(\bar{F}\) is a homological functor.

  Let \(G\colon\Tri\to\Ab\) be any homological functor.  If \(B\inOb\Sub\), then \(\id_B\colon B\to B\) is a final object in \(\Subo{B}\), so that we get \(\Rightco G(B) \cong G(B)\).  In particular, this shows that \(\bar{F}|_\Sub=F\).  Hence \(\Rightco(\Rightco G) \cong \Rightco G\) and~\(\Rightco G\) is colocal.  The definition of the colocalisation~\(\Rightco G\) shows that two colocal functors that agree on~\(\Sub\) already agree on all of~\(\Tri\).

  If~\(H\) is any colocal homological functor with a natural transformation \(H\Rightarrow G\), then we get an induced natural transformation \(H\cong \Rightco H\Rightarrow \Rightco G\).  It is straightforward to see that this provides a bijection between natural transformations \(H\Rightarrow G\) and \(H\Rightarrow \Rightco G\).
\end{proof}

\section{The localisation--colocalisation exact sequence}
\label{sec:exact_sequence}

To relate the localisation and colocalisation of~\(\Tri\) at~\(\Sub\), we introduce a third filtered category~\(\Tria{B}\) that combines \(\Wo{B}\) and \(\Subo{B}\) for an object~\(B\) of~\(\Tri\).  Objects of \(\Tria{B}\) are exact triangles of the form
\[
E \to B\xrightarrow{s} C\to E[1]
\]
with \(E\inOb\Sub\) or, equivalently, \(s\in\Qeq\); arrows in~\(\Tria{B}\)
are morphisms of triangles of the form
\[
\xymatrix{
  E\ar[d]\ar[r]&B\ar[r]^s\ar@{=}[d]&C\ar[r]\ar[d]&E[1]\ar[d]\\
  E'\ar[r]& B\ar[r]^{s'}&C'\ar[r]&E'[1].}
\]

There are obvious forgetful functors from \(\Tria{B}\) to \(\Wo{B}\) and~\(\Subo{B}\) that extract the map \(B\to C\) or the map \(E\to B\), respectively.  Since \(s\in\Qeq\) if and only if \(E\inOb\Sub\), any object of \(\Wo{B}\) or \(\Subo{B}\) is in the range of this forgetful functor.  The axiom~(TR3) for triangulated categories implies that these two forgetful functors are surjective on arrows as well.

\begin{prop}
  \label{prop:category_triangles_filtered}
  The category~\(\Tria{B}\) is filtered.
\end{prop}

\begin{proof}
  First we check~(F1).  Consider two exact triangles
  \[
  \Delta \defeq \bigl(E \to B\xrightarrow{s} C\to E[1]\bigr),\qquad
  \Delta' \defeq \bigl(E' \to B\xrightarrow{s'} C'\to E'[1]\bigr)
  \]
  with \(E,E'\inOb\Sub\) and \(s,s'\in\Qeq\).  (LF2) yields a commuting diagram
  \[
  \xymatrix{B\ar[r]^{s'}\ar[d]^s&
    C'\ar[d]^{t}\\
    C\ar[r]^g &C''}
  \]
  with \(t\in\Qeq\).  Then \(ts'\in\Qeq\) as well.  We embed~\(ts'\) in an exact triangle
  \[
  \Delta'' \defeq \bigl(E'' \to B\xrightarrow{ts'} C''\to E''[1]\bigr);
  \]
  this yields an object of~\(\Tria{B}\).  The axioms of a triangulated
  category yield morphisms of triangles
  \[
  \xymatrix{E\ar[r]\ar[d]&B\ar[r]^s\ar@{=}[d]&C\ar[r]\ar[d]^g&E[1]\ar[d]\\
    E''\ar[r]&B\ar[r]^{ts'}&C''\ar[r]&E''[1],}
  \qquad
  \xymatrix{E'\ar[r]\ar[d]&B\ar[r]^{s'}\ar@{=}[d]&C'\ar[r]\ar[d]^t&E'[1]\ar[d]\\
    E''\ar[r]&B\ar[r]^{ts'}&C''\ar[r]&E''[1].}
  \]
  Thus~\(\Delta''\) dominates both \(\Delta\) and~\(\Delta'\) in~\(\Tria{B}\), verifying~(F1).

  Next we construct equalisers for parallel arrows
  in~\(\Tria{B}\):
  \begin{equation}
    \begin{gathered}
      \xymatrix{
        E\ar[r]^g\ar@<-0.5ex>[d]_-{\varepsilon_1} \ar@<0.5ex>[d]^-{\varepsilon_2}\ar[r]&
        B\ar[r]^s\ar@{=}[d]&
        C\ar[r]^f\ar@<-0.5ex>[d]_-{\gamma_1}\ar@<0.5ex>[d]^-{\gamma_2}&
        E[1]\ar@<-0.5ex>[d] \ar@<0.5ex>[d]\\
        E'\ar[r]^{g'}&
        B\ar[r]^{s'}&
        C'\ar[r]^{f'}&E'[1].
      }
    \end{gathered}
  \end{equation}
  Since the category~\(\Wo{B}\) is filtered, it is easy to equalise \(\gamma_1\) and~\(\gamma_2\) by a morphism in~\(\Wo{B}\).  This lifts to a morphism in~\(\Tria{B}\) that equalises \(\gamma_1\) and~\(\gamma_2\).  Since the category~\(\Subo{B}\) is filtered as well by Proposition~\ref{lem:Subo_filtered}, we may equalise \(\varepsilon_1\) and~\(\varepsilon_2\) by a morphism in~\(\Subo{B}\), which once again lifts to a morphism in~\(\Tria{B}\) that still equalises \(\varepsilon_1\) and~\(\varepsilon_2\).  By~(F1), we may equalise both at the same time, that is, we get a commuting diagram
  \[
  \xymatrix{
    E\ar[r]^g\ar@<-0.5ex>[d]_-{\varepsilon_1} \ar@<0.5ex>[d]^-{\varepsilon_2}\ar[r]&
    B\ar[r]^s\ar@{=}[d]&
    C\ar[r]^f\ar@<-0.5ex>[d]_-{\gamma_1}\ar@<0.5ex>[d]^-{\gamma_2}&
    E[1]\ar@<-0.5ex>[d] \ar@<0.5ex>[d]\\
    E'\ar[r]^{g'}\ar[d]^{\varepsilon_3}&
    B\ar[r]^{s'}\ar@{=}[d]&
    C'\ar[r]^{f'}\ar[d]^{\gamma_3}&E'[1]\ar[d]\\
    E''\ar[r]^{g'}&
    B\ar[r]^{s'}&
    C''\ar[r]^{f'}&E''[1]
  }
  \]
  with \(\gamma_3\gamma_1=\gamma_3\gamma_2\) and \(\varepsilon_3\varepsilon_1=\varepsilon_3\varepsilon_2\).  Thus~\(\Tria{B}\) satisfies~(F2).
\end{proof}

Since the forgetful functors from~\(\Tria{B}\) to \(\Wo{B}\) and~\(\Subo{B}\) are surjective both on objects and arrows, they preserve colimits.  That is, we may rewrite localisations and colocalisations as colimits over~\(\Tria{B}\).

\begin{thm}
  \label{thm:localisation_colocalisation_sequence}
  Let~\(\Tri\) be a triangulated category and~\(\Sub\) a thick subcategory.  Let \(F\colon \Tri\to \Ab\) be a homological functor to the category of Abelian groups.  Then there is a natural exact sequence
  \[
  \dotsb
  \to \Rightco F_1(B)
  \to  F_1(B)
  \to \Right F_1(B)
  \to \Rightco F_0(B)
  \to  F_0(B)
  \to \Right F_0(B)
  \to \dotsb
  \]
\end{thm}

This is the exact sequence~\eqref{eq:derived_exact_sequence} promised in the introduction.

\begin{proof}
  For each object \(E\to B\to C\to E[1]\) of~\(\Tria{B}\), we get a long exact sequence
  \begin{equation}
    \label{eq:derived_exact_sequence_2}
    \dotsb \to F_1(E) \to F_1(B) \to F_1(C) \to F_0(E) \to F_0(B) \to F_0(C) \to \dotsb.
  \end{equation}
  This is a functor from \(\Tria{B}\) to the category of exact chain complexes in~\(\Ab\).  Since~\(\Tria{B}\) is filtered, the colimit of this diagram of chain complexes is again an exact chain complex.

  Since the forgetful functors from \(\Tria{B}\) to \(\Wo{B}\) and \(\Subo{B}\) are surjective on objects and arrows, we have
  \[
  \Right F(B) \cong \varinjlim_{\Tria{B}} F(C),\qquad
  \Rightco F(B) \cong \varinjlim_{\Tria{B}} F(E).
  \]
  Since~\(\Tria{B}\) is filtered, the colimit of the constant diagram~\(F(B)\) on~\(\Tria{B}\) is \(F(B)\).  Hence the colimit of the exact sequences~\eqref{eq:derived_exact_sequence_2} is the desired exact sequence~\eqref{eq:derived_exact_sequence}.
\end{proof}

Theorem~\ref{thm:localisation_colocalisation_sequence} for the functor \(\Tri(A,\blank)\) implies the exact sequence~\eqref{eq:localisation_exact_sequence}.

When we apply Theorem~\ref{thm:localisation_colocalisation_sequence} to a homological functor \(\Tri^\op\to\Ab\), we get the exact sequence~\eqref{eq:derived_exact_sequence} for a cohomological functor \(\Tri\to\Ab\).

\begin{cor}
  \label{cor:local_colocal_criterion}
  A homological or cohomological functor~\(F\) is local if and only if \(\Rightco F\cong0\), and~\(F\) is colocal if and only if \(\Right F\cong0\).
\end{cor}

\begin{cor}
  \label{cor:invertibility_local_colocal}
  A natural transformation \(\Phi\colon F\Rightarrow F'\) between two homological or cohomological functors is invertible if and only if both its localisation \(\Right\Phi\colon \Right F\Rightarrow \Right F'\) and colocalisation \(\Rightco\Phi\colon \Rightco F\Rightarrow \Rightco F'\) are invertible.
\end{cor}

\begin{proof}
  It is clear that invertibility of~\(\Phi\) implies invertibility of \(\Right\Phi\) and~\(\Rightco\Phi\).  The converse follows from the Five Lemma and Theorem~\ref{thm:localisation_colocalisation_sequence}.
\end{proof}

\end{document}